\documentclass[10pt,a4paper]{article}
\usepackage{amsmath}
\usepackage{amssymb}
\usepackage{amsthm}
\usepackage{amsfonts}
\usepackage{enumerate}
\usepackage{pstricks}
\usepackage{pst-plot}
\usepackage{pst-poly}
\usepackage{graphics} 
\usepackage{graphicx}
\usepackage{pgfpages}
\usepackage{algorithm}
\usepackage{algorithmic}
\usepackage{url}
\usepackage{booktabs}
\usepackage{verbatim}
\usepackage[margin=1.3in]{geometry}
\newcommand{\tluste}[1]{\mbox{\mathversion{bold}$ #1 $}}

\newcommand{\imace}[1]{\mbox{$\tluste{#1}$}}

\def\Mid#1{{#1^c}}
\def\Rad#1{#1^\Delta}

\newcommand{\onum}[1]{\mbox{$\overline{{#1}}$}} 
\newcommand{\unum}[1]{\mbox{$\underline{{#1}}$}}

\newcommand{\ivr}[1]{\mbox{$\tluste{#1}$}} 

\newcommand{\inum}[1]{\mbox{$\tluste{#1}$}}

\newcommand{\R}[0]{{\mathbb{R}}}

\def\eps{{\varepsilon}}

\newcommand{\mmid}[0]{;\,}		

\newcommand{\seznam}[1]{{\{1, \ldots, {#1}\}}}

\newcommand{\st}[0]{{\ \ \mbox{subject to}\ \ }}

\DeclareMathOperator{\diag}{diag}	

\def\nref#1{$(\ref{#1})$}

\newtheorem{proposition}{Proposition}

\newtheorem{conjecture}{Conjecture}
\theoremstyle{definition}

\newtheorem{example}{Example}

\begin{document}

\title{On the Efficient Gerschgorin Inclusion Usage in the Global Optimization $\alpha$BB Method}

\author{
  Milan Hlad\'{i}k\footnote{
Charles University, Faculty  of  Mathematics  and  Physics,
Department of Applied Mathematics, 
Malostransk\'e n\'am.~25, 11800, Prague, Czech Republic, 
e-mail: \texttt{milan.hladik@matfyz.cz}
}
}

\date{\today}
\maketitle

\begin{abstract}
In this paper, we revisit the $\alpha$BB method for solving global optimization problems. We investigate optimality of the scaling vector used in Gerschgorin's inclusion theorem to calculate bounds on the eigenvalues of the Hessian matrix. We propose two heuristics to compute good scaling vector~$d$, and state three necessary optimality conditions for optimal~$d$. Since the scaling vector calculated by the second presented method satisfies all three optimality conditions, it serves as a cheap but efficient solution.
\end{abstract}


\section{Introduction}

Interval branch \& bound is a deterministic global optimization method to rigorously enclose (under some general assumptions) the optimal solutions and optimal values by arbitrarily tight intervals. The basic idea is to include the feasible set in an initial box (hyperinterval), and then iteratively split the box into sub-boxes and eliminate the idle ones not containing optimal solutions.
Depending on the concrete splitting strategy, relaxation of nonlinear or nonconvex terms, and other features we distinguish different approaches \cite{FloAkr2005,Flo2000,HanWal2004,HenTot2010,Kea1996,Kea2011,KreKub2010,Neu2004,NinMes2011}.

We focus particularly on the $\alpha$BB method \cite{AdjDal1998, AdjAnd1998, AndMar1995, Flo2000, FloGou2009, FloPar2009}, which is based on a convex relaxation in the following way.
Let $f:\R^n\mapsto\R$ be a twice-differentiable objective or constraint function and $x_i\in\inum{x}_i=[\unum{x}_i,\onum{x}_i]$, $i=1,\dots,n$, interval domains for the variables. 
The aim is to construct a function $g:\R^n\mapsto\R$ satisfying two conditions:
\begin{enumerate}[(1)]
\item
$f(x)\geq g(x)$ for every $x\in\ivr{x}$,
\item
$g(x)$ is convex on $x\in\ivr{x}$.
\end{enumerate}
The classical global optimization $\alpha$BB method utilizes the convex underestimator in the form of
\begin{align}\label{fceUnder}
g(x):=f(x)-\sum_{i=1}^n\alpha_i(\onum{x}_i-x_i)(x_i-\unum{x}_i),
\end{align}
where $\alpha_i\geq0$, $i=1,\dots,n$, are determined such that $g(x)$ is convex. The Hessian of $g(x)$ reads
$$
\nabla^2 g(x)=\nabla^2 f(x)+2\diag(\alpha),
$$
where $\diag(\alpha)$ is the diagonal matrix with entries $\alpha_1,\dots,\alpha_n$. There are several modifications of this method. The generalization from \cite{AkrMey2004,SkjWes2012} considers convex underestimators in the form of
$$
g(x):=f(x)-(\onum{x}-x)^TP(x-\unum{x})+q,
$$
where $P\in\R^{n\times n}$ is a symmetric matrix with a non-negative diagonal and $q\in\R$ is a correction value calculated so that the underestimation property holds true. Naturally, when $P$ is a diagonal matrix and $q=0$, the underestimator reduces to \nref{fceUnder}. 
Another class of underestimators defined as
$$
g(x):=f(x)-\sum_{i=1}^n 
 (1-e^{\gamma_i(\onum{x}_i-x_i)})(1-e^{\gamma_i(x_i-\unum{x}_i)})
$$
was discussed in \cite{AkrFlo2004,AkrFlo2004b,FloPar2009}, yielding the so called $\gamma$BB method. Theoretical justification for  $\alpha$BB and $\gamma$BB relaxation terms is given in \cite{FloKre2007}.
A global optimization method QBB based on convex underestimators and branch \& bound scheme on simplices, was proposed in \cite{ZhuKun2005}.
Other convex and linear relaxations were investigated in \cite{Ans2012,DomNeu2012,ScoStu2011}, for instance. 

Let us consider the classical $\alpha$BB approach utilizing the form \nref{fceUnder}. The parameters $\alpha_i$-s may be calculated it the following way.
Let $\imace{H}$ be an interval matrix enclosing the image of $\nabla^2 f(x)$ over $x\in\ivr{x}$. That is, the $(i,j)$th element of $\imace{H}$ is an interval $\inum{h}_{ij}=[\unum{h}_{ij},\onum{h}_{ij}]$ such that
$$
\frac{\partial^2}{\partial x_i \partial x_j}f(x)\in\inum{h}_{ij},\quad
\forall x\in\ivr{x}.
$$
Now, to achieve convexity of $g(x)$, it is sufficient to choose $\alpha$ such that each matrix in $\imace{H}+2\diag(\alpha)$ is positive semidefinite, i.e., its eigenvalues are non-negative.
Eigenvalues of interval matrices were investigated e.g.\ in \cite{AdjDal1998,Flo2000,Hla2013a,HlaDan2011b,HlaDan2010,Mon2011}. For the purpose of the $\alpha$BB method, it seems that the most convenient method for bounding eigenvalues of interval matrices is the scaled Gerschgorin inclusion \cite{AdjDal1998,AdjAnd1998,Flo2000}. Its benefits are that it is easy to compute and eliminate the unknowns $\alpha_i$, $i=1,\dots,n$, and it is also usually sufficiently tight. For any positive $d\in\R^n$, we can put
\begin{align}\label{alpha}
\textstyle
\alpha_i:=\max\left\{0, -\frac{1}{2}
 \left(\unum{h}_{ii}-\sum_{j\not=i}|\inum{h}_{ij}|d_j/d_i\right)\right\},
\quad i=1,\dots,n,
\end{align}
where $|\inum{h}_{ij}|=\max\left\{|\unum{h}_{ij}|,|\onum{h}_{ij}|\right\}$. To reflect the range of the variable domains, it is recommended to use $d=\Rad{x}:=\frac{1}{2}(\onum{x}-\unum{x})$. The efficiency of symbolic computation of the Hessian matrix is studied in \cite{Hla2013b}.

In this paper, we investigate optimality of the choice of $d$, and propose some heuristics to achieve better scaling vector $d$ for which the corresponding vector $\alpha$ have less conservative overestimation. We state three optimality conditions for $d$ and show that the second presented method satisfies all of them.

\section{Computation of $\alpha$ and $d$}

In this section, we study computation of $\alpha$ from \nref{alpha} and optimal or nearly optimal choice of the scaling vector $d$.

\subsection*{Optimal choice of $d$}

Since we employ only one endpoint of each interval in $\imace{H}$, it is sufficient to consider the point matrix $H\in\R^{n\times n}$ defined as
\begin{align*}
h_{ii}&:=\unum{h}_{ii}\ \ \forall i,\\
h_{ij}&:=-\max(|\unum{h}_{ij}|,|\onum{h}_{ij}|),\ \ \forall i\not=j.
\end{align*}
Now, \nref{alpha} takes the form of
\begin{align}\label{alpha2}
\textstyle
\alpha_i:=\max\left\{0, -\frac{1}{2}
 \left(h_{ii}+\sum_{j\not=i}h_{ij}d_j/d_i\right)\right\},
\quad i=1,\dots,n.
\end{align}
We will assume that $H$ is not block diagonal (even after simultaneous permutation of rows and columns) since otherwise we can process the particular blocks independently.

Computation of $\alpha$ from  \nref{alpha2} depends on the scaling vector $d>0$. In \cite{AdjDal1998,Flo2000}, the authors propose to use $d:=\Rad{x}$ to reflect the widths of interval domains of variables. Below, we show that this choice is optimal in some sense.

\begin{proposition}\label{thmOptAplhaNeg}
Suppose that \nref{alpha2} is satisfied as equation without using the positive part for $d=\Rad{x}$.
Then the maximum separation distance between $f$ and its underestimator is minimized for $d=\Rad{x}$.
\end{proposition}

\begin{proof}
The maximum separation distance can be expressed as
$$
\max \sum_{i=1}^n\alpha_i(\onum{x}_i-x_i)(x_i-\unum{x}_i)\st x\in\ivr{x}.
$$
It is easy to see (cf.\ \cite{SkjWes2012}) that the maximum is attained for $x=\Mid{x}$, so the distance is
$$
\sum_{i=1}^n\alpha_i(\Rad{x_i})^2.
$$
Substituting for $\alpha$ yields a lower bound on this distance
$$
\sum_{i=1}^n-\frac{1}{2}
 \left({h}_{ii}+\sum_{j\not=i}{h}_{ij}\frac{d_j}{d_i}\right)(\Rad{x_i})^2.
$$
To minimize this function over $d>0$, we neglect the absolute term and obtain
$$
-\frac{1}{2}\sum_{i=1}^n \sum_{j\not=i}{h}_{ij}(\Rad{x_i})^2\frac{d_j}{d_i}.
$$
Now, we express $d_i$, $i=1,\dots,n$, as $d_i=\Rad{x_i}c_i$, and have
\begin{align}\label{pfSum}
-\frac{1}{2}\sum_{i=1}^n \sum_{j\not=i}
 {h}_{ij}\Rad{x_i}\Rad{x_j}\frac{c_j}{c_i}
=-\frac{1}{2}\sum_{i<j}{h}_{ij}
   \Rad{x_i}\Rad{x_j}\left(\frac{c_i}{c_j}+\frac{c_j}{c_i}\right).
\end{align}
Notice that $c_i/c_j+c_j/c_i\geq2$ for every $c_i,c_j>0$, and the inequality is satisfied as equation for $c_i=c_j=1$. Thus, $c=1$ is the minimizer of each term of \nref{pfSum}, and therefore it minimizes the whole function as well. Hence, $d=\Rad{x}$ is an optimum for the original problem as well since it is feasible solution that minimizes the lower estimation function.
\end{proof}

\subsection*{Local improvement I.}

There is a non-trivial class of problems satisfying the assumption of Proposition~\ref{thmOptAplhaNeg} --- for instance problems with ${h}_{ii}\leq0$ $\forall i$.

If the assumption of Proposition~\ref{thmOptAplhaNeg} is not satisfied, we can still compute the optimal choice for $d$ by solving an auxiliary optimization problem
\begin{align}\label{minDist}
\min\sum_{i=1}^n\alpha_i(\Rad{x_i})^2
\st \alpha_i\geq0,\ 
2\alpha_i\geq-{h}_{ii}-\sum_{j\not=i}{h}_{ij}\frac{d_j}{d_i},\ 
d_i\geq1 \ \forall i.
\end{align}
To solve it requires some computational effort. Thus, one can think of a local improvement of $d$, which is set up as $d:=\Rad{x}$ at the beginning.
When the assumption of Proposition~\ref{thmOptAplhaNeg} does not hold, then there is at least one dominant diagonal entry. That is, there is $i\in\seznam{n}$ such that
\begin{align}\label{ineqDom}
{h}_{ii}+\sum_{j\not=i}{h}_{ij}\frac{d_j}{d_i}>0.
\end{align}
We say that $i$-th row of $H$ is not saturated.
Decreasing the value of $d_i$ by a sufficiently small amount will preserve $\alpha_i$ to be zero, but possibly decreases values of $\alpha_j$ for $j\not=i$. It is easy to see that the smallest value of $d_i$ we can put is 
\begin{align}\label{putDiImpr}
d_i:=-\frac{1}{{h}_{ii}}\sum_{j\not=i}{h}_{ij}d_j.
\end{align}
If we do this adjusting for each $i\in\seznam{n}$ satisfying \nref{ineqDom}, then it may happen that some new diagonal entry becomes dominant. Thus, we repeat the process until no improvement happens, or until ${h}_{ii}+\sum_{j\not=i}{h}_{ij}d_j/d_i\geq0$ for every $i$; the latter means that each matrix in $\imace{H}$ is positive semidefinite and so $f(x)$ is convex. Each iteration of the presented heuristic is cheap, but may significantly reduce the overestimation of $\alpha_i$s.

\begin{example}
Consider the function
$$
f(x)=5x_1^{}x_2^2+\frac{100}{3}x_1^3-\frac{7}{6}x_2^3,
$$
where $x_1,x_2\in[1,2]$. Its Hessian reads
\begin{align*}
\nabla^2 f(x)=
\begin{pmatrix}
200x_1 & 10x_2 \\ 10x_2 & 10x_1-7x_2
\end{pmatrix}
\end{align*}
and its (tightest) interval enclosure is computed as
\begin{align*}
\nabla^2 f(\ivr{x})\subseteq\imace{H}=
\begin{pmatrix}
[200, 400] & [10,20] \\  [10,20] &  [-4,13]
\end{pmatrix},\ \mbox{ whence }\ 
H=\begin{pmatrix}200 & -20 \\ -20 & -4\end{pmatrix}.
\end{align*}
Choosing $d:=2\Rad{x}=(1,1)^T$, we arrive at the value of $\alpha=(0,12)$.

Now, let us proceed along the local improvement method. Since
$$
{h}_{11}+\sum_{j\not=i}{h}_{1j}\frac{d_j}{d_1}=180\geq0,
$$
we can modify $d_1$ according to \nref{putDiImpr} as $d_1:=0.1$. For the new scaling vector $d=(0.1,1)$ we calculate  $\alpha=(0,3)$. Thus, the improvement in $\alpha_2$ is significant.
\end{example}

As the following example shows, there are however situations, when the improvements are very small and the number of iterations may be potentially infinite.

\begin{example}\label{exOscil}
Consider the matrix
\begin{align*}
H=
\begin{pmatrix}
8& -1& -6\\-1& -2& 0\\-6& 0& 6
\end{pmatrix}.
\end{align*}
Let $\Rad{x}=(1,1,1)^T$. Thus, the initial feasible solution is $d=(1,1,1)^T$, $\alpha=(0,3,0)^T$, and the corresponding objective of \nref{minDist} is 3. Since the first matrix row is not saturated, we may decrease $d_1$ to obtain $d=(\frac{7}{8},1,1)^T$, and the objective value $\frac{23}{8}$. Now, the third row is not saturated, leading to the improvement $d=(\frac{7}{8},1,\frac{7}{8})^T$, and the objective value $2+\frac{7}{8}$. This improvements iterates; in $k$-th iteration we have 
$$
d=\left(\tfrac{1}{2}(1+(\tfrac{3}{4})^k),1,\tfrac{1}{2}(1+(\tfrac{3}{4})^k)\right)^T,
$$
and the corresponding the objective value is $2+\tfrac{1}{2}(1+(\tfrac{3}{4})^k)$. The limit is the vector $d=(\tfrac{1}{2},1,\tfrac{1}{2})^T$ with objective $2+\tfrac{1}{2}$.
\end{example}

\subsection*{Local improvement II.}

Example~\ref{exOscil} motivates us to improve the local improvement method in such a way that several (or all) constraints satisfying \nref{ineqDom} are processed together when decreasing $d_i$ for the other variables.

\begin{proposition}\label{propNonneg}
Suppose that the optimal value to \nref{minDist} is positive.
If there is an optimal solution to \nref{minDist}, then there is one such that
\begin{align}\label{ineqPropNonneg}
{h}_{ii}+\sum_{j\not=i}{h}_{ij}\frac{d_j}{d_i}\leq0,
\quad \forall i=1,\dots,n.
\end{align}
\end{proposition}

\begin{proof}
Let $d^0,\alpha^0$ be an optimal solution, and \nref{ineqPropNonneg} be violated for $i\in I\subseteq\seznam{n}$. Now, we extend (iteratively) the index set $I$ of those $i$ for which
\begin{align*}
{h}_{ii}+\sum_{j\not=i}{h}_{ij}\frac{d^0_j}{d^0_i}=0,
\quad \forall i=1,\dots,n,
\end{align*}
and there is $j\in I$ such that ${h}_{ij}\not=0$. Notice that for $i\in I$ we have $\alpha^0_i=0$.

Let $H_I$ be the submatrix of $H$ when restricted to the rows and columns indexed by $I$, and similarly let $d_I$ be the subvector of $d$ indexed by $I$. Denoting $a$ the vector with entries
$$
-\sum_{j\not\in I}{h}_{ij}d^0_j, \quad i\in I,
$$ 
we have
$$
H_Id^0_I\geq a\geq 0.
$$
If $H_Id^0_I>a$ in addition, then $H_I$ in an M-matrix, and so it has an entrywise nonnegative inverse. Denote by $d^*$ the vector $d^0$ for which $d^0_I$ is replaced by the solution of the linear system $H_Id_I=a$. By the above observation, $d_I^*$ is nonnegative and satisfies
$$
d^0_I\geq H_I^{-1}a=d_I^*.
$$
That is, the corresponding $\alpha_i$-s for $d^*$ remain zero for $i\in I$, and the other $\alpha_i$-s either decrease or remain the same. It would be wrong if some entry of $d_I^*$ is zero, but we will show that it cannot happen. Suppose to the contrary that $d^*_i=0$ for $i\in I^0\subseteq I$. Since $d^*$ satisfies the equations
$$
{h}_{ii}d^*_i+\sum_{j\not=i}{h}_{ij}d^*_j=0,
\quad \forall i\in I^0,
$$
we obtain
$$
\sum_{j\not\in I}{h}_{ij}d^*_j=0,
\quad \forall i\in I^0,
$$
from which  ${h}_{ij}=0$, $\forall i\in I^0$, $	\forall j\not\in I^0$. Thus, $H$ is either block diagonal, or positive semidefinite (when $d^*=0$), which contradicts our assumption.

If $H_Id^0_I>a$ is not satisfied, yet the strict inequality still holds for at least one index $i$ due to the assumption. Decreasing $d^0_i$ by a sufficiently small amount for every such $i$ will not violate the strict inequality, but it causes the equations in $H_Id^0_I\geq a$ to hold as strict inequalities. Thus, $H_I$ is an M-matrix, and the previous results is valid here, too. 
\end{proof}

Notice that the assumption of Proposition~\ref{propNonneg} on positivity of the optimal value to \nref{minDist} very weak. Indeed, it is satisfied for every $\imace{H}$ containing at least one matrix that is not positive semidefinite.

Further, its proof gives rise to another local improvement algorithm: Let $d$ be a feasible solution to \nref{minDist}, identify $I$, set up $H_I$ and $a$ and update $d_I$ by $H_I^{-1}a$. The scheme of the method is given in Algorithm~\ref{alg2}; the inequalities in expressions are understood entry-wise.

\begin{algorithm}[ht]\label{alg2}
\caption{(Local improvement II.)}
\begin{algorithmic}[1]
\STATE
Put $d:=\Rad{x}$;
\WHILE{\textbf{not} ($Hd\leq0$ \textbf{or} $Hd\geq0$)}
\STATE
identify $I$; 
\STATE
put $a=(-\sum_{j\not\in I}{h}_{ij}d^0_j)_{i\in I}$;
\STATE
put $d_I:=H_I^{-1}a$;
\ENDWHILE
\RETURN{$d$};
\end{algorithmic}
\end{algorithm}

The update of $d_I$ may cause that other row of $H$ becomes unsaturated and enters $I$. This is why more than one iteration is needed in general. However, contrary to the previous method, this situation can happen at most $(n-1)$-times, which gives an upper bound on the number of iterations.
The following Examples~\ref{exOscil2}--\ref{exLin} respectively show that the performance to the first local improvement method is much higher, but the maximum number of $n-1$ iterations may be sometimes attained.

\begin{example}\label{exOscil2}
For the data of Example~\ref{exOscil} with the initial point $d=(1,1,1)^T$, we set up $I:=\{1,3\}$, and update 
$d_I:=\left(\begin{smallmatrix}8&-6\\-6&6\end{smallmatrix}\right)^{-1}
 \left(\begin{smallmatrix}1\\0\end{smallmatrix}\right)
=(0.5,0.5)^T$. Thus, we get the optimum $d^*=(0.5,1,0.5)^T$ in only one iteration.
\end{example}

\begin{example}\label{exLin}
Consider a tridiagonal matrix $H\in\R^{n\times n}$, where 
$$
h_{11}=2,\quad h_{nn}=0,\quad
h_{kk}=4\frac{2^{2k-2}-1}{2^{2k-1}-1}\ \ \forall k=2,\dots,n-1,
$$
and $h_{ij}=-1$ for $|i-j|=1$.
Let $d=(1,\dots,1)^T$.
Then in the first iteration, $d_1$ is updated to $d_1=\frac{1}{2}$. In iteration $k$, we have $I=\{k\}$, and $d_k$ is updated to $1-2^{1-2k}$. We prove it by induction since $d_{k+1}$ is calculated as
\begin{align*}
d_{k+1}&:=-\frac{d_kh_{k+1,k}+d_{k+2}h_{k+1,k+2}}{h_{k+1,k+1}}
=\frac{d_k+1}{h_{k+1,k+1}}\\
&=\frac{2-2^{1-2k}}{4}\cdot\frac{2^{2k+1}-1}{2^{2k}-1}
=2^{-2k-1}(2^{2k+1}-1)
=1-2^{1-2(k+1)}.
\end{align*}
This makes $k+1$ enter to $I$ because of the inequality
$$
h_{k+1,k+1}d_{k+1}+h_{k+1,k}d_k+h_{k+1,k+2}d_{k+2}\geq0,
$$
which takes the form
$$
4\frac{2^{2k}-1}{2^{2k+1}-1}-(1-2^{1-2k})-1\geq0.
$$
Multiplying by $(2^{2k+1}-1)/2$ we get an easy-to-see inequality
\begin{equation*}
2(2^{2k}-1)-2(1-2^{-2k})(2^{2k+1}-1)/2=
2(2^{2k}-1)-(-1+2^{2k+1}+2^{-2k}-2)\geq0.
\end{equation*}
\end{example}

\begin{example}\label{exTest}
We did some small experiments about what is the number of iterations. We generated the entries of $H\in\R^{n\times n}$ randomly as integers in $[-10,10]$ with uniform distribution. Then the diagonal of $H$ was increased by $n$. Table~\ref{tabAv} displays the average number and the maximal number of iterations as the mean of 10000 runs. We considered only such matrices for which there is at least one iteration needed (the others were skipped), that is, $H$ is not positive semidefinite and has at least one unsaturated row.

Next, the table also show similar characteristics for tridiagonal matrices. Herein, the tridiagonal entries are generated in the same manner, but without the diagonal increment. These matrices have almost the same behavior as the general ones.

The table shows that for the randomly generated matrices the number of iterations is never larger than three, and mostly it is only one.

\begin{table}[t]
\caption{(Example~\ref{exTest}) Number of iterations for random matrices.\label{tabAv}}
\begin{center}
\begin{tabular}{rcccc}
 \toprule
$n$ & \multicolumn{2}{c}{general} &\multicolumn{2}{c}{tridiagonal}\\
\cmidrule(lr){2-3}
\cmidrule(lr){4-5}
 & {average} & {maximal} & {average} & {maximal} \\
\midrule 
 3 & 1.0587  & 2 & 1.0254 &2 \\
 5 & 1.1491  & 3 & 1.0499 &2 \\
10 & 1.0506  & 3 & 1.0751 &3 \\
15 & 1.0063  & 2 & 1.0941 &3 \\
20 & 1.0009  & 2 & 1.1285 &3 \\
\bottomrule
\end{tabular}
\end{center}
\end{table}
\end{example}

\subsection*{Optimality conditions}

By Proposition~\ref{propNonneg}, we can rewrite \nref{minDist} as 
\begin{align*}
\min\sum_{i=1}^n\alpha_i(\Rad{x_i})^2
\st \alpha_i\geq0,\ 
2\alpha_i\geq-{h}_{ii}-\sum_{j\not=i}{h}_{ij}\frac{d_j}{d_i}\geq0,\ 
d_i\geq1 \ \forall i,
\end{align*}
or,
\begin{align}\label{minDist2}
\min\frac{1}{2}\sum_{i=1}^n(\Rad{x_i})^2
\left(-{h}_{ii}-\sum_{j\not=i}{h}_{ij}\frac{d_j}{d_i}\right)
\st
-\sum_{j=1}^n{h}_{ij}d_j\geq0,\ 
d_i\geq1 \ \forall i.
\end{align}
This formulation reveals new properties. We state three necessary optimality conditions below.

In the remainder of the paper, we will without loss of generality assume that $\Rad{x}=1$; otherwise we proceed as in the proof of Proposition~\ref{thmOptAplhaNeg}. 

For a positive vector $d$, we define
$$
I^*=I^*(d):=\{i\in\seznam{n}\mmid
 -\textstyle\sum_{j=1}^n{h}_{ij}d_j>0\}.
$$

\begin{proposition}\label{propDEqual}
Let $d$ be an optimal solution to \nref{minDist2}.
Then $d_i=c\Rad{x}_i$, $i\in I^*$, for some constant~$c$.
\end{proposition}

\begin{proof}
Since we assume $\Rad{x}=1$, we have to show that $d_i$ are the same for all $i\in I^*$.

Suppose to the contrary that this is not the case. For the sake of simplicity assume that there is $k\in I^*$ such that $d_k>d_i$ $\forall i\in I^*\setminus\{k\}$. If the dominant value is not unique, we will gather them in the following way. Let $K\subsetneqq I^*$ be the index set of these dominant values and define a new matrix $H'$ of size $n-|K|+1$, indexed by $0$ and $i\in K':=\seznam{n}\setminus K$, and having entries
\begin{align*}
h'_{ij}&=h_{ij},\quad i,j\in K',\\
h'_{i0}&=h_{0i}=\sum_{k\in K}{h}_{ik},\quad i\in K',\\
h'_{00}&=\sum_{j,k\in K,j\not=k}{h}_{jk}.
\end{align*}
Omitting the absolute terms and the multiplicative constant, the objective function of \nref{minDist2} reads
\begin{align*}
\sum_{i=1}^n\sum_{j\not=i}{h}_{ij}\frac{d_j}{d_i}
&=\sum_{i\not\in K}\sum_{j\not\in K,j\not=i}{h}_{ij}\frac{d_j}{d_i}
 +\sum_{i\not\in K}\sum_{j\in K}{h}_{ij}\frac{d_j}{d_i}
 +\sum_{i\in K}\sum_{j\not\in K}{h}_{ij}\frac{d_j}{d_i}
 +\sum_{i\in K}\sum_{j\in K,j\not=i}{h}_{ij}\frac{d_j}{d_i}\\
&=\sum_{i\not=0}\sum_{j\not=0,j\not=i}h'_{ij}\frac{d'_j}{d'_i}
 +\sum_{i\not=0}h'_{i0}\frac{d'_0}{d'_i}
 +\sum_{j\not=0}{h}_{0j}\frac{d'_j}{d'_0}
 +h'_{00}
=\sum_{i\in K'\cup\{0\}}\sum_{j\not=i}h'_{ij}\frac{d'_j}{d'_i}+h'_{00},
\end{align*}
where $d'_i=d_i$, $i\in K'$, and $d'_0=d_k$ for any $k\in K$.
This reduces the problem to one with the matrix $H'$, the feasible solution $d'$, and satisfying the uniqueness condition.

Let $\eps>0$ be sufficiently small, and consider a variation $\tilde{d}$ of $d$ defined as 
$\tilde{d}_k:=(1-\eps)d_k$ and $\tilde{d}_i:=d_i$ for $i\not=k$. Since some rows of $H$ may not be saturated now, we apply the procedure of the proof of Proposition~\ref{propNonneg} to obtain a solution satisfying \nref{ineqPropNonneg}; the procedure concerns only $\tilde{d}_i$, $i\not\in I^*$.
For the same reason as in the proof of Proposition~\ref{propNonneg} the entries of $\tilde{d}$ remain positive, and $\tilde{d}_i\leq d_i$ $\forall i\not\in I^*$.

We show that the objective value in $\tilde{d}$ is smaller than in $d$ by evaluating their (double) difference
\begin{align*}
\sum_{i}\left(-{h}_{ii}-\sum_{j\not=i}{h}_{ij}\frac{d_j}{d_i}\right)
-\sum_{i}
 \left(-{h}_{ii}-\sum_{j\not=i}{h}_{ij}\frac{\tilde{d}_j}{\tilde{d}_i}\right)
=\sum_{i\in I^*}\sum_{j\not=i}
{h}_{ij}\left(\frac{\tilde{d}_j}{\tilde{d}_i}-\frac{d_j}{d_i}\right).
\end{align*}
We will split this expression into a sum $A+B$ and show nonnegativity of both sub-expressions $A$ and $B$. The first one,
\begin{align*}
A
&:=\sum_{i=k}\sum_{j\in I^*}
{h}_{ij}\left(\frac{\tilde{d}_j}{\tilde{d}_i}-\frac{d_j}{d_i}\right)
+\sum_{i\in I^*\setminus\{k\}}\sum_{j=k}
{h}_{ij}\left(\frac{\tilde{d}_j}{\tilde{d}_i}-\frac{d_j}{d_i}\right)\\
&=\sum_{j\in I^*\setminus\{k\}}
{h}_{kj}\left(\frac{d_j}{(1-\eps)d_k}-\frac{d_j}{d_k}\right)
+\sum_{i\in I^*\setminus\{k\}}
{h}_{ik}\left(\frac{(1-\eps)d_k}{d_i}-\frac{d_k}{d_i}\right)\\
&=\sum_{j\in I^*\setminus\{k\}} {h}_{kj}
  \frac{d_j}{d_k}\cdot\frac{\eps}{1-\eps}
 -\sum_{j\in I^*\setminus\{k\}} \eps {h}_{kj}\frac{d_k}{d_j}\\
&=\sum_{j\in I^*\setminus\{k\}} \eps{h}_{kj}\left(
 \frac{d_j}{d_k}\cdot\frac{1}{1-\eps}
 -\frac{d_k}{d_j}\right).
\end{align*}
By definition, $d_k>d_j$ $\forall j\in I^*\setminus\{k\}$, whence $(1-\eps)d_k^2>d_j^2$. This implies $\frac{d_j}{d_k}\cdot\frac{1}{1-\eps}
 <\frac{d_k}{d_j}$, and due to ${h}_{kj}\leq0$ we get $A>0$. Notice that the case $A=0$ cannot happen since $H$ is not block diagonal.

Now, the remainder reads
\begin{align*}
B&:=
\sum_{i=k}\sum_{j\not\in I^*}
 {h}_{ij}\left(\frac{\tilde{d}_j}{\tilde{d}_i}-\frac{d_j}{d_i}\right)
+\sum_{i\in I^*\setminus\{k\}}\sum_{j\not=k}
{h}_{ij}\left(\frac{\tilde{d}_j}{\tilde{d}_i}-\frac{d_j}{d_i}\right)\\
&=\sum_{j\not\in I^*}
 {h}_{kj}\left(\frac{\tilde{d}_j}{(1-\eps)d_k}-\frac{d_j}{d_k}\right)
+\sum_{i\in I^*\setminus\{k\}}\sum_{j\not\in I^*}
{h}_{ij}\left(\frac{\tilde{d}_j}{d_i}-\frac{d_j}{d_i}\right)\\
\end{align*}
Using the facts that
\begin{itemize}
\item
$(1-\eps)d_k^2>d_i^2$ for every $i\in I^*\setminus\{k\}$, 
\item
$ {h}_{ij}(\tilde{d}_j-d_j)\geq0$ for every $j\not\in I^*$, 
\end{itemize}
we derive
\begin{align*}
(1-\eps)d_k^2 B 
&\geq\sum_{j\not\in I^*}
  {h}_{kj}d_k(\tilde{d}_j-(1-\eps)d_j)
 +\sum_{j\not\in I^*}\sum_{i\in I^*\setminus\{k\}}
  {h}_{ij}d_i(\tilde{d}_j-d_j)\\
&=\sum_{j\not\in I^*}\sum_{i\in I^*} {h}_{ij}d_i\tilde{d}_j
 -\sum_{j\not\in I^*}\left({h}_{kj}(1-\eps)d_kd_j+
   \sum_{i\in I^*\setminus\{k\}} {h}_{ij}{d}_i d_j\right)\\
&=\sum_{j\not\in I^*}\sum_{i\in I^*} {h}_{ij}d_i\tilde{d}_j
 -\sum_{j\not\in I^*}\sum_{i\in I^*} {h}_{ij}\tilde{d}_i d_j.
\end{align*}
Since
\begin{align*}
\sum_{i=1}^nh_{ji}d_i=\sum_{i=1}^nh_{ji}\tilde{d}_i=0
\end{align*}
for any $j\not\in I^*$, we can write the last expression as
\begin{align*}
\sum_{j\not\in I^*}\sum_{i\in I^*} {h}_{ij}d_i\tilde{d}_j
 +\sum_{j\not\in I^*}\sum_{i\not\in I^*} {h}_{ij}\tilde{d}_i d_j
=\sum_{j\not\in I^*}\sum_{i=1}^n {h}_{ij}d_i\tilde{d}_j
=\sum_{j\not\in I^*}0\tilde{d}_j=0,
\end{align*}
which completes the proof.
\end{proof}

\begin{proposition}\label{propInewInonI}
Let $d$ be an optimal solution to \nref{minDist2}.
Then $d^{}_i\Rad{x}_i\leq d^{}_j\Rad{x}_j$ for every $i\not\in I^*$ and $j\in I^*$.
\end{proposition}

\begin{proof}
Since we assume $\Rad{x}=1$, we have to show that $d^{}_i\leq d^{}_j$ for every $i\not\in I^*$ and $j\in I^*$.

Suppose to the contrary that there is $k\not\in I^*$ such that $d_k>d_i$, $i\in I^*$. 

Let $\eps>0$ be sufficiently small, and consider a variation $\tilde{d}$ of $d$ defined as 
$\tilde{d}_k:=(1-\eps)d_k$ and $\tilde{d}_i:=d_i$ for $i\not=k$. Since some rows of $H$ may not be saturated now, we apply the procedure of the proof of Proposition~\ref{propNonneg} again to obtain a solution satisfying \nref{ineqPropNonneg}; the procedure concerns only $\tilde{d}_i$, $i\not\in I^*\cup\{k\}$. Notice that positivity of $\tilde{d}$ is ensured since principal submatrix of an M-matrix is again an M-matrix for the same reason as in the proof of Proposition~\ref{propNonneg}.

Recall that by Proposition~\ref{propDEqual}, we have $d_i=\,$const for $i\in I^*$.
We show that the objective in $\tilde{d}$ is smaller than in $d$ by evaluating their (double) difference.
\begin{align*}
\sum_{i}\left(-{h}_{ii}-\sum_{j\not=i}{h}_{ij}\frac{d_j}{d_i}\right)
-\sum_{i}
 \left(-{h}_{ii}-\sum_{j\not=i}{h}_{ij}\frac{\tilde{d}_j}{\tilde{d}_i}\right)
=\sum_{i\in I^*\cup\{k\}}\sum_{j\not=i}
{h}_{ij}\left(\frac{\tilde{d}_j}{\tilde{d}_i}-\frac{d_j}{d_i}\right)
=\frac{1}{d_i^2}(A+B),
\end{align*}
where
\begin{align*}
A&:=d_i^2\sum_{i\in I^*}\sum_{j\not=i}
  {h}_{ij}\left(\frac{\tilde{d}_j}{\tilde{d}_i}-\frac{d_j}{d_i}\right)
 =d_i^2\sum_{i\in I^*}\sum_{j\not\in I^*}
  {h}_{ij}\left(\frac{\tilde{d}_j}{\tilde{d}_i}-\frac{d_j}{d_i}\right)\\
 &=d_i^2\sum_{i\in I^*}\sum_{j\not\in I^*}
  {h}_{ij}\left(\frac{\tilde{d}_j}{d_i}-\frac{d_j}{d_i}\right)
 = \sum_{j\not\in I^*}\left(\tilde{d}_j\sum_{i\in I^*} {h}_{ij}d_i
    -d_j\sum_{i\in I^*} {h}_{ij}\tilde{d}_i\right)\\
 &=\sum_{j\not\in I^*}\left(\tilde{d}_j\sum_{i\in I^*} {h}_{ij}d_i
    +d_j\sum_{i\not\in I^*} {h}_{ij}\tilde{d}_i\right)
    -d_k\sum_{i\not\in I^*} {h}_{ik}\tilde{d}_i
    -d_k\sum_{i\in I^*} {h}_{ik}\tilde{d}_i\\
 &=\sum_{j\not\in I^*}\tilde{d}_j\sum_{i=1}^n {h}_{ij}d_i
    -d_k\sum_{i=1}^n {h}_{ij}\tilde{d}_i
 = -d_k\sum_{i=1}^n {h}_{ij}\tilde{d}_i.
\end{align*}
We used the facts that $\sum_{i=1}^n {h}_{ij}\tilde{d}_i=0$ $\forall j\not\in I^*\cup\{k\}$ and $\sum_{i=1}^n {h}_{ij}d_i=0$ $\forall j\not\in I^*$.
The second term in the above expression reads
\begin{align*}
B&:=d_i^2\sum_{j\not=k}
  {h}_{kj}\left(\frac{\tilde{d}_j}{\tilde{d}_k}-\frac{d_j}{d_k}\right)
 =d_i^2\left(\sum_{j\not=k}
  {h}_{kj}\frac{\tilde{d}_j}{\tilde{d}_k}-h_{kk}\right)\\
&=\frac{d_i^2}{\tilde{d}_k}\sum_{j=1}^n {h}_{kj}\tilde{d}_j
 >d_k\sum_{j=1}^n {h}_{kj}\tilde{d}_j.
\end{align*}
The strict inequality follows from $d_i^2<d_k\tilde{d}_k=(1-\eps)d_k^2$ and $h_{kj}\leq0$, $j\not=k$. Notice that not all $h_{kj}$, $j\not=k$, can be zero since $H$ is not block diagonal.

Now, it is clear that $A+B>0$, a contradiction.
\end{proof}

The last optimality condition presented says that no row with the dominant diagonal entry, that is an originally unsaturated row, can enter $I^*$.

\begin{proposition}\label{propDiagDom}
Let $k\in\seznam{n}$ such that $\sum_{j=1}^nh_{kj}\Rad{x}_i\geq0$. Then $k\not\in I^*(d)$ for any optimal solution $d$ to \nref{minDist2}.
\end{proposition}

\begin{proof}
Let $d$ be an optimal solution to \nref{minDist2}, and assume that $\Rad{x}=1$. 
Suppose to the contrary that there is $k$ such that  $\sum_{j=1}^nh_{kj}\geq0$ and $\sum_{j=1}^nh_{kj}d_j<0$.
By Propositions~\ref{propDEqual} and~\ref{propInewInonI} we know that $d_k=d_i$ $\forall i\in I^*$, and $d_k\geq d_i$ $\forall i\not\in I^*$. Hence
\begin{align*}
0>\sum_{j=1}^nh_{kj}d_j\geq \sum_{j=1}^nh_{kj}d_k\geq0,
\end{align*}
a contradiction.
\end{proof}

Notice that Algorithm~\ref{alg2} yields solutions that satisfy all necessary optimality conditions given by Propositions~\ref{propDEqual}--\ref{propDiagDom}. In our numerical experiments, we have found no solution that would not be optimal. This justifies us to state it as a conjecture.

\begin{conjecture}
Algorithm~\ref{alg2} yields optimal solutions to \nref{minDist2}.
\end{conjecture}


\section{Conclusion}

We discussed the optimal choice of the scaling vector used in the $\alpha$BB method. We proposed two local improvement heuristics. In particular, the second one is promising since it satisfies all three necessary conditions that we stated. This led us to conjecture that the method yields always optimal solutions. Further, the method not only runs in polynomial time, but also, as indicated by our numerical experiments, the average number of iterations is not much greater than one.

\subsubsection*{Acknowledgments.} 

The author was supported by the Czech Science Foundation Grant P402/13-10660S.


\bibliographystyle{abbrv}
\bibliography{alpha_comp}

\end{document}